\documentclass[preprint,12pt]{elsarticle}

\usepackage{amsfonts, amsmath, amscd}
\usepackage[psamsfonts]{amssymb}

\usepackage{amssymb}

\newtheorem{theorem}{Theorem}
\newtheorem{lemma}[theorem]{Lemma}
\newdefinition{rmk}{Remark}
\newproof{proof}{Proof}
\newtheorem{pro}[theorem]{Proposition}
\newtheorem{defin}[theorem]{Definition}
\newtheorem{cor}[theorem]{Corollary}
\newtheorem{problem}[theorem]{Problem}
\newdefinition{exa}{Example}

\numberwithin{equation}{section}
\numberwithin{theorem}{section}
\numberwithin{exa}{section}

\begin{document}

\begin{frontmatter}

\title{Pontryagin duality for Abelian $s$- and $sb$-groups\tnoteref{label1}}\tnotetext[label1]{The author was partially supported  by Israel Ministry of Immigrant Absorption}

\author{S.S. Gabriyelyan}
\ead{saak@math.bgu.ac.il}

\address{Department of Mathematics, Ben-Gurion University of the
Negev, Beer-Sheva, P.O. 653, Israel}

%%%%%%%%
\begin{abstract}
The main goal of the article is to study the Pontryagin duality for Abelian $s$- and $sb$-groups.  Let $G$ be an infinite Abelian group and $X$ be the dual group of the discrete group $G_d$. We show that a dense subgroup $H$ of $X$ is $\mathfrak{g}$-closed iff $H$ algebraically is the dual group of $G$ endowed with  some maximally almost periodic $s$-topology. Every reflexive Polish Abelian group is $\mathfrak{g}$-closed in its Bohr compactification. If a $s$-topology $\tau$ on a countably infinite Abelian group $G$ is generated by a countable set of convergent sequences, then the dual group of $(G,\tau)$ is Polish. A non-trivial Hausdorff Abelian topological group is a $s$-group iff it is a quotient group of the $s$-sum of a family of copies of $(\mathbb{Z}^\mathbb{N}_0, \mathbf{e})$.

\end{abstract}

\begin{keyword}

$T$-sequence\sep  $TB$-sequence\sep  Abelian group\sep  $s$-group\sep  $sb$-group\sep  dual group\sep  $\mathfrak{g}$-closed subgroup\sep  sequentially covering map

\MSC[2008] 22A10 \sep 22A35 \sep 43A05 \sep 43A40
\end{keyword}

\end{frontmatter}

\section{Introduction}

{\bf I. Notations and preliminaries result.}  A group $G$ with the discrete topology is denoted by $G_{d}$. The subgroup generated by a subset $A$ of $G$ is denoted by $\langle A\rangle$. Let $X$ be an Abelian topological group. A basis of open neighborhoods at zero of  $X$ is denoted by $\mathcal{U}_X$.  The group of all continuous characters on $X$ is denoted by $\widehat{X}$. $\widehat{X}$ endowed with the compact-open topology is denoted by $X^{\wedge}$.  Denote by $\mathbf{n}(X) = \cap_{\chi\in \widehat{X}} {\rm ker} \chi$ the von Neumann radical of $X$. If $\mathbf{n}(X) = \{ 0\}$, $X$ is called maximally almost periodic ($MAP$).

Let $X$ be an Abelian topological group and $\mathbf{u} =\{ u_n \}$ be a sequence of elements of $\widehat{X}$. Following Dikranjan et al. \cite{DMT}, we denote by $s_{\mathbf{u}} (X)$ the set of all $x\in X$ such that $(u_n , x)\to 1$. Let $H$ be a subgroup of $X$. If $H=s_{\mathbf{u}} (X)$ we say that $\mathbf{u}$ {\it characterizes} $H$ and that $H$ is {\it characterized} (by $\mathbf{u}$) \cite{DMT}. Let $X$ be a metrizable compact Abelian group. By \cite[Corollary 1]{Ga1}, each characterized subgroup $H= s_{\mathbf{u}} (X)$ admits a locally quasi-convex Polish group topology. We denote the group $H=s_{\mathbf{u}} (X)$  with this topology by $H_{\mathbf{u}}$.

Let $H$ be a subgroup of an Abelian topological group $X$. Following \cite{DMT}, the closure operator $\mathfrak{g}_X$ is defined as follows
\[
\mathfrak{g} (H)= \mathfrak{g}_X (H) := \bigcap_{ \mathbf{u} \in \widehat{X}^\mathbb{N} } \left\{ s_{\mathbf{u}} (X) : \; H \leq s_{\mathbf{u}} (X) \right\},
\]
and we say that $H$ is $\mathfrak{g}$-{\it closed} if $H=\mathfrak{g}(H)$. For an arbitrary subset $S$ of $\widehat{X}^\mathbb{N}$, one puts
\[
s_S (X) := \bigcap_{\mathbf{u}\in S} s_\mathbf{u} (X).
\]

Let $\mathbf{u}=\{ u_n\}$ be a non-trivial sequence in an Abelian  group $G$. The following  very important questions has been studied by many authors as  Graev \cite{Gra}, Nienhuys \cite{Nie}, and others:
\begin{problem} \label{prob1}
{\it Is there a Hausdorff group topology $\tau$ on $G$ such that $u_n \to 0$ in $(G,\tau)$}?
\end{problem}
Protasov and  Zelenyuk \cite{ZP1, ZP2} obtained a criterion that gives the complete answer to this question. Following  \cite{ZP1}, we say that a sequence $\mathbf{u} =\{ u_n \}$ in a group $G$ is a $T$-{\it sequence} if there is a Hausdorff group topology on $G$ in which $u_n $ converges to $0$. The group $G$ equipped with the finest Hausdorff group topology $\tau_\mathbf{u}$ with this property is denoted by $(G, \mathbf{u})$.  A $T$-sequence $\mathbf{u}=\{ u_n\}$ is called trivial if there is $n_0$ such that $u_n =0$ for every $n\geq n_0$.

Let $G$ be a countably infinite Abelian group, $X=G^\wedge_d$, $\mathbf{u} =\{ u_n \}$ be a $T$-sequence in $G$ and $H=s_\mathbf{u} (X)$. There is a simple dual connection between the groups $(G, \mathbf{u})$ and $H_\mathbf{u}$, and, moreover, we can compute the von Neumann radical $\mathbf{n} (G,\mathbf{u})$ of $(G,\mathbf{u})$ as follows:
\begin{theorem} \label{t01} {\rm \cite{Ga1}}
 $(G, \mathbf{u})^{\wedge} = H_{\mathbf{u}}$ and, algebraically, $\mathbf{n} (G,\mathbf{u}) = H^\perp$.
\end{theorem}

The counterpart of Problem \ref{prob1} for {\it precompact} group topologies on $\mathbb{Z}$ is studied by Raczkowski \cite{Rac}. Following \cite{BDM} and motivated by \cite{Rac}, we say that a sequence $\mathbf{u} =\{ u_n\}$ is a $TB$-{\it sequence} in an Abelian group $G$ if there is a {\it precompact} Hausdorff group topology on $G$ in which  $u_n \to 0$.  The group $G$ equipped with the finest precompact Hausdorff group topology $\tau_{b\mathbf{u}}$ with this property is denoted by $(G, b\mathbf{u})$.

For an Abelian group $G$ and an arbitrary subgroup $H\leq G_d^\wedge$,  let $T_H$ be the weakest topology on $G$ such that all characters of $H$ are continuous with respect to $T_H$. One can easily show \cite{CoR} that $T_H$ is a totally bounded group topology on $G$, and it is Hausdorff iff $H$ is dense in $G_d^\wedge$.

A subset $A$ of a topological space $\Omega$ is called {\it sequentially open} if whenever a sequence $\{ u_n\} $ converges to a point of $A$, then  all but finitely many of the members $u_n$ are contained in $A$. The space $\Omega$ is called {\it sequential} if any subset $A$ is open if and only if $A$ is sequentially open.  Franklin \cite{Fra} gave  the following characterization of sequential spaces:
\begin{theorem} \label{t001} {\rm \cite{Fra}}
A topological space is sequential if and only if it is a quotient of a metric space.
\end{theorem}

The following natural generalization of Problem \ref{prob1}
was considered in \cite{Ga3}:
\begin{problem} \label{prob2}
{\it Let $G$ be a group and $S$ be a set of sequences in $G$. Is there a (resp. precompact) Hausdorff group topology $\tau$ on $G$ in which every sequence of $S$ converges to zero}?
\end{problem}

By analogy with $T$- and $TB$-sequences, we define \cite{Ga3}:
\begin{defin} \label{d01}
{\it Let $G$ be an Abelian group and $S$ be a set of sequences in $G$. The set $S$ is called a $TS$-set (resp. $TBS$-set) of sequences if there is a Hausdorff (resp. precompact Hausdorff) group topology on $G$ in which all sequences of $S$ converge to zero. The finest Hausdorff (resp. precompact Hausdorff) group topology with this property is denoted by $\tau_S$ (resp. $\tau_{bS}$). }
\end{defin}

The set of all $TS$-sets (resp. $TBS$-sets) of sequences of a group $G$ we denote by $\mathcal{TS} (G)$ (resp. $\mathcal{TBS} (G)$).
It is clear that, if $S\in \mathcal{TS}(G)$ (resp. $S\in \mathcal{TBS}(G)$), then $S' \in \mathcal{TS}(G)$ (resp. $S' \in \mathcal{TBS}(G)$) for every nonempty subset $S'$ of $S$ and every sequence $\mathbf{u}\in S$ is a $T$-sequence (resp. $\mathbf{u}$ is a $TB$-sequence). Evidently, $\tau_S \subseteq \tau_{S'}$ (resp. $\tau_{bS} \subseteq \tau_{bS'}$). Also, if $S$ contains only trivial $T$-sequences, then $S\in \mathcal{ST}(G)$ and $\tau_S$ is discrete.

By definition, $\tau_\mathbf{u}$ is finer than $\tau_S$ (resp. $\tau_{b\mathbf{u}}$ is finer than $\tau_{bS}$) for every $\mathbf{u} \in S$. Thus, if $U$ is open (resp. closed) in $\tau_S$, then it is open (resp. closed) in $\tau_\mathbf{u}$ for every $\mathbf{u}\in S$.
So, by definition, we obtain that $\tau_S \subseteq \bigwedge_{\mathbf{u}\in S} \tau_\mathbf{u}$ (resp. $\tau_{bS} \subseteq \bigwedge_{\mathbf{u}\in S} \tau_{b\mathbf{u}}$).

The following class of topological groups is defined in \cite{Ga3}:
\begin{defin}
{\it A Hausdorff Abelian topological group $(G, \tau)$ is called a $s$-group (resp. a $bs$-group) and the topology $\tau$ is called a $s$-topology (resp. a $bs$-topology) on $G$ if there is $S\in \mathcal{TS}(G)$ (resp. $S\in \mathcal{TBS}(G)$) such that $\tau =\tau_S$ (resp. $\tau =\tau_{bS}$). }
\end{defin}
In other words, $s$-groups are those topological groups whose topology can be described by a set of convergent sequences. The family of all  Abelian $s$-group is denoted by  $\mathbf{SA}$.

One of the most natural way how to find $TS$-sets of sequences is as follows. Let $(G, \tau)$ be a Hausdorff Abelian topological group. We denote the set of all  sequences of $(G, \tau)$ converging to zero  by $S(G,\tau)$:
\[
S(G,\tau) =\left\{ \mathbf{u} =\{ u_n \} \subset G : \; u_n \to 0 \mbox{ in } \tau \right\}.
\]
It is clear, that $S(G,\tau)\in \mathcal{TS}(G)$ and $\tau \subseteq \tau_{S(G,\tau)}$. The group $\mathbf{s}(G,\tau):= (G, \tau_{S(G,\tau)})$ is called the $s$-refinement of $(G, \tau)$ \cite{Ga3}.

In \cite{Ga3} it is proved that the class $\mathbf{SA}$ is closed under taking of quotient and it is finitely multiplicative. It is natural that this class contains all sequential groups \cite{Ga3}. For every countable $TS$-set of sequences in an Abelian group $G$ the space $(G,\tau_S)$ is complete and sequential (see \cite{Ga3}). Another non-trivial examples of sequential Hausdorff Abelian groups see \cite{ChMPT}. A complete description of Abelian $s$-groups is given in \cite{Ga3}.

Let $X$ and $Y$ be topological groups. Following Siwiec \cite{Siw}, a continuous homomorphism $p: X\to Y$ is called {\it sequence-covering} if and only if it is surjective and for every sequence $\{ y_n \}$ converging to the unit $e_Y$ there is a sequence $\{ x_n \}$ converging to $e_X$ such that $p(x_n) = y_n$.

{\bf II. Main results.} The main goal of the article is to study the Pontryagin duality for Abelian $s$- and $sb$-groups.  We  give a simple dual connection between the $MAP$ $s$-topologies on an infinite Abelian group $G$ and the dense $\mathfrak{g}$-closed subgroups of the compact group $G_d^\wedge$. Also we describe all $bs$-topologies on $G$.

The article is organized as follows.
In Section \ref{sec1} we study  the dual groups of  Abelian $s$- and $sb$-groups  and prove the following generalization  of the algebraic part of Theorem \ref{t01}:
\begin{theorem} \label{t1}
Let $S\in \mathcal{TS}(G)$ for an  infinite Abelian group $G$ and $i_S : G_d \to (G, \tau_S), i_S (g)=g,$ be the natural continuous isomorphism. Then
\begin{enumerate}
\item[{\rm 1)}] $i_S^\wedge \left((G, \tau_S)^\wedge \right) =s_S (G_d^\wedge)$;
\item[{\rm 2)}] $\mathbf{n} (G,\tau_S) =\left[  s_S (G^\wedge_d ) \right]^\perp$ algebraically.
\end{enumerate}
\end{theorem}
Also, in this section, we describe all  $sb$-topologies on an infinite Abelian group $G$. In \cite{DMT}, it was pointed out that $\tau_{b\mathbf{u}} = T_{s_\mathbf{u} (G_d^\wedge)}$ for one $TB$-sequence $\mathbf{u}$. The following theorem generalizes this fact:
\begin{theorem} \label{t2}
{\it Let $S\in \mathcal{TBS}(G)$ for an  infinite Abelian group $G$ and $j_S : G_d \to (G, \tau_{bS}), j_S (g)=g,$ be the natural continuous isomorphism. Then
\begin{enumerate}
\item[{\rm 1)}] $j_S^\wedge \left((G, \tau_S)^\wedge \right) =s_S (G_d^\wedge)$;
\item[{\rm 2)}] $\tau_{bS} = T_{s_S (G_d^\wedge)}$.
\end{enumerate} }
\end{theorem}
As an immediate corollary of Theorems \ref{t1} and \ref{t2} we obtain:
\begin{cor} \label{c01}
{\it  Let $S\in \mathcal{TBS}(G)$ for an  infinite Abelian group $G$ and $j : (G, \tau_{S}) \to (G, \tau_{bS}), j(g)=g,$ be the natural continuous isomorphism. Then its conjugate homomorphism $j^\wedge : (G, \tau_{bS})^\wedge \to (G, \tau_{S})^\wedge$ is a continuous isomorphism.}
\end{cor}

Using Theorem \ref{t1} we obtain  the following dual connection between  dual groups  of $s$-groups and $\mathfrak{g}$-closed subgroups of  compact Abelian groups:
\begin{theorem} \label{t12}
Let $G$ be an infinite Abelian group. Set $X= G_d^{\wedge}$.
\begin{enumerate}
\item[{\rm (i)}] If $S\in \mathcal{TS}(G)$ and $i_S : G_d \to (G,\tau_S)$ is the natural continuous isomorphism, then $i^\wedge_S \left( (G, \tau_S )^\wedge \right)$ is a $\mathfrak{g}$-closed subgroup of $X$.
\item[{\rm (ii)}] If $H$ is a $\mathfrak{g}$-closed subgroup of $X$, then there is $S\in \mathcal{TBS}\left( ({\rm cl} H)^\wedge \right)$ such that $H=\left( ({\rm cl} H)^\wedge, \tau_S \right)^\wedge$ algebraically.
\end{enumerate}
\end{theorem}

The following two statements are immediately corollaries of Theorem \ref{t12} and the fact that every sequential group is a $s$-group \cite{Ga3}:
\begin{cor} \label{c02}
A dense subgroup $H$ of an infinite compact Abelian group $X$  is $\mathfrak{g}$-closed if and only if $H$ algebraically is the dual group of $\widehat{X}$ endowed with some $MAP$ $s$-topology.
\end{cor}

\begin{cor} \label{c03}
The dual group of a sequential group $(G,\tau)$ is a $\mathfrak{g}$-closed subgroup of the compact group $G_d^\wedge$.
\end{cor}

For  an Abelian topological group $(G,\tau)$, $bG = \widehat{G}_d^\wedge$ denotes its Bohr compactification. We shall identify $G$, if it is $MAP$, and $G^{\wedge\wedge}$ with their images in $bG$ (see below Section \ref{sec1}). From Corollary \ref{c02} we obtain:
\begin{cor} \label{c04}
Let an Abelian topological group $(G,\tau)$ be such that $G^\wedge$ is a $s$-group. Then $G^{\wedge\wedge}$ is a dense $\mathfrak{g}$-closed subgroup of $bG$.
\end{cor}

Further, we show:
\begin{pro} \label{p01}
Every reflexive Polish Abelian group (in particular, every  separable locally convex Banach space or separable metrizable locally compact Abelian group) is $\mathfrak{g}$-closed in its Bohr compactification.
\end{pro}

In \cite{Ga3}, a general criterion to be a $s$-group is given. In Section \ref{sec2} we obtain another  analog of Franklin's theorem \ref{t001} for  Abelian $s$-groups.
Let $\{ G_i \}_{i\in I}$, where $I$ is a non-empty set of indices, be a family of Abelian groups. The direct sum of  $G_i$ is denoted by
\[
\sum_{i\in I} G_i :=\left\{ (g_i)_{i\in I} \in \prod_{i\in I} G_i : \; g_i = 0 \mbox{ for almost all } i \right\}.
\]
We denote by $j_k $ the natural including of $G_k$ into $\sum_{i\in I} G_i$, i.e.:
\[
j_k (g)=(g_i)\in \sum_{i\in I} G_i, \mbox{ where } g_i =g  \mbox{ if } i=k \mbox{ and } g_i =0  \mbox{ if } i\not= k.
\]
Let $G_i =(G_i, \tau_i)$ be an Abelian $s$-group for every $i\in I$. It is easy to show that the set $\bigcup_{i\in I} j_i \left( S(G_i, \tau_i)\right)$
is a $TS$-set of sequences in $\sum_{i\in I} G_i$ (see Section \ref{sec3}).
\begin{defin} \label{d04}
{\it Let $\{ (G_i, \tau_i)\}_{i\in I}$ be a non-empty family of Abelian $s$-groups. The group $\sum_{i\in I} G_i$ endowed with the finest Hausdorff group topology $\tau^s$ in which every sequence of $\bigcup_{i\in I} j_i \left( S(G_i, \tau_i)\right)$ converges to zero is called the {\bf $s$-sum} of $G_i$ and it is denoted by ${s-\sum}_{i\in I} G_i$.}
\end{defin}
By definition, the $s$-sum of $s$-groups is a $s$-group either. Note that the $s$-sum of $s$-groups can be defined also for non-Abelian $s$-groups.

Set $\mathbb{Z}^{\mathbb{N}}_0 =\{ (n_1,\dots, n_k, 0,\dots) | n_j \in \mathbb{Z} \}$ and $\mathbf{e} =\{ e_n\} \in\mathbb{Z}_0^{\mathbb{N}}$, where $e_1 =(1,0,0,\dots), e_2 = (0,1,0,\dots), \dots$. Then $\mathbf{e}$ is a $T$-sequence in $\mathbb{Z}^{\mathbb{N}}_0 $. The following theorem  gives a characterization of Abelian $s$-groups and it can be considered as a natural analog of Franklin's theorem \ref{t001}:
\begin{theorem} \label{t06}
Let $(X,\tau)$ be a non-discrete Hausdorff Abelian topological group. The following statements are equivalent:
\begin{enumerate}
\item[{\rm (i)}] $(X,\tau)$ is a $s$-group;
\item[{\rm (ii)}] $(X,\tau)$ is a quotient group of the $s$-sum of a non-empty family of copies of $(\mathbb{Z}_0^\mathbb{N} , \mathbf{e})$. Moreover, a quotient map may be chosen to be sequence-covering.
\end{enumerate}
\end{theorem}

Let $G$ be an infinite Abelian group.  In Section \ref{sec3} we consider the case of countable $S\in \mathcal{TS}(G)$. In this case,  the topology $\tau_S$ has a simple description (see Proposition \ref{p1}). The main result of the section is the following:
\begin{theorem} \label{t4}
{\it Let $G$ be a countably infinite Abelian group and let  $S=\{ \mathbf{u}_n\}_{n\in\omega} \in \mathcal{TS}(G)$.  Then $(G,\tau_S)^\wedge$ is a Polish group. More precisely, $(G,\tau_S)^\wedge$ embeds onto a closed subgroup of the Polish group $\prod_{n\in\omega} (G,\mathbf{u}_n)^\wedge$. }
\end{theorem}
As a corollary we prove the following two propositions (see Problems 2.21 and 2.22 \cite{Ga2}):
\begin{pro} \label{p2}
Let $\{ X_n\}_{n\in\omega} $ be a sequence of second countable locally compact Abelian groups. Then there is a complete countably infinite Abelian $MAP$ $s$-group $(G,\tau)$ such that
\[
(G,\tau)^\wedge = \prod_{n\in\omega} X_n.
\]
\end{pro}

\begin{pro} \label{p3}
There is a complete sequential $MAP$ group topology $\tau$ on $\mathbb{Z}^{\mathbb{N}}_0$ such that
\[
(\mathbb{Z}^{\mathbb{N}}_0 ,\tau)^\wedge = \mathbb{R}^{\mathbb{N}}.
\]
\end{pro}

In the last section we pose some open questions.

\section{Duality} \label{sec1}

The following lemma will be used several times in the article:
\begin{lemma} {\rm \cite[Lemma 3.1]{DMT}} \label{l11}
{\it Let $G$ be an Abelian topological group and let $H\leq G^\wedge$. Then, for a sequence $\mathbf{u} =\{ u_n \}$ in $G$, one has $u_n \to 0$ in $(G, T_H)$ if and only if $H\leq s_\mathbf{u} (G^\wedge)$. }
\end{lemma}

The  following proposition connects the notions of $T$- and $TB$-sequences (for countably infinite $G$ see \cite{Ga1}).
\begin{pro} \label{p1}
{\it Let $\mathbf{u} =\{ u_n \}$ be a sequence in an Abelian group $G$. Then
\begin{enumerate}
\item[{\rm (i)}] $\mathbf{u}$ is a $TB$-sequence if and only if  it is a $T$-sequence and $(G, \mathbf{u})$ is $MAP$.
\item[{\rm (ii)}]  Let $\mathbf{u}$ be a $TB$-sequence and  let  $i_\mathbf{u} :G_d \to (G,\mathbf{u})$ and $j_\mathbf{u} : G_d \to (G, b\mathbf{u})$, $i_\mathbf{u}(g)=j_\mathbf{u} (g)=g$, be the natural continuous isomorphisms. Then $i_\mathbf{u}^\wedge \left((G, \mathbf{u})^\wedge\right) = j_\mathbf{u}^\wedge \left((G, b\mathbf{u})^\wedge\right) = s_\mathbf{u} (G_d^\wedge)$.
\item[{\rm (iii)}]  {\rm \cite{DMT}}  If $\mathbf{u}$ is a $TB$-sequence, then $\tau_{b\mathbf{u}} =T_{s_\mathbf{u} (G_d^\wedge)}$.
\end{enumerate} }
\end{pro}

\begin{proof}
Set $H=i_\mathbf{u}^\wedge \left( (G,\mathbf{u})^\wedge \right)$,  $Y=j_\mathbf{u}^\wedge \left( (G,b\mathbf{u})^\wedge \right)$ and $h:(G,\mathbf{u}) \to (G, b\mathbf{u})$, $h(g)=g$. Then $h$ is a continuous isomorphism and $j_\mathbf{u} =h\circ i_\mathbf{u}$. So $j_\mathbf{u}^\wedge = i_\mathbf{u}^\wedge \circ h^\wedge$ and $Y\subseteq H$.

(i) It is clear that, if a sequence $\mathbf{u}=\{ u_n \}$ is a $TB$-sequence, then it is a $T$-sequence and $(G, \mathbf{u})$ is $MAP$. Let us prove the converse assertion.  Let $x\in H$ and $x=i_\mathbf{u}^\wedge (\chi), \chi \in (G,\mathbf{u})^\wedge$. Then $(u_n, x)=(i_\mathbf{u} (u_n), \chi)  \to 1$. Thus $Y\subseteq H\subseteq s_\mathbf{u} (G_d^\wedge)$. Hence, by Lemma \ref{l11}, $u_n \to 0$ in $T_{H}$.  Since $i_\mathbf{u}^\wedge$ is injective and $(G,\mathbf{u})$ is $MAP$, the topology $T_{H}$ is Hausdorff. Since $T_H$ is precompact,  $\mathbf{u}$ is a $TB$-sequence.

(ii) We claim that $Y=H=s_\mathbf{u} (G_d^\wedge)$. Indeed, by  Lemma \ref{l11}, $u_n \to 0$ in $T_{s_\mathbf{u} (G_d^\wedge)}$. Thus the topology $\tau_{b\mathbf{u}}$ is finer than $T_{s_\mathbf{u} (G_d^\wedge)}$. Hence, by \cite[1.2 and 1.4]{CoR}, we have $s_\mathbf{u} (G_d^\wedge) \subseteq Y$. By item (i) of the proof, we obtain that $Y=H=s_\mathbf{u} (G_d^\wedge)$.

(iii)  follows from item (ii) and \cite[Theorem 1.2]{CoR}.
\end{proof}

{\it Proof of Theorem} \ref{t1}.
1) By definition, the natural inclusions $i_\mathbf{u} :G_d \to (G,\mathbf{u})$ and $t_\mathbf{u} : (G,\mathbf{u})\to (G, \tau_S)$, $i_\mathbf{u}(g)=t_\mathbf{u} (g)=g$, are continuous isomorphisms for every $\mathbf{u}\in S$, and $i_S =t_\mathbf{u} \circ i_\mathbf{u}$. By Proposition \ref{p1}(ii), $i_\mathbf{u}^\wedge \left( (G, \mathbf{u})^\wedge \right) = s_\mathbf{u} (G^\wedge_d )$. Hence $i_S^\wedge \left((G, \tau_S)^\wedge \right) \subseteq s_\mathbf{u} (G_d^\wedge)$ for every $\mathbf{u}\in S$. So $i_S^\wedge \left((G, \tau_S)^\wedge \right) \subseteq s_S (G_d^\wedge)$.

Conversely, let $x\in s_S (G_d^\wedge)$. By Proposition \ref{p1}(ii), $x\in i_\mathbf{u}^\wedge \left((G, \mathbf{u})^\wedge \right)$ for every $\mathbf{u}=\{ u_n\} \in S$. Thus, $x$ is an algebraic homomorphism from $(G, \tau_S)$ into $\mathbb{T}$ such that, by the definition of the topology $\tau_\mathbf{u}$, $(u_n , x)\to 1$  for every $\mathbf{u} \in S$. By \cite[Theorem 2.4]{Ga3}, $x$ is a continuous character of $(G, \tau_S)$. So $x\in i_S^\wedge \left((G, \tau_S)^\wedge \right)$.

2) By 1), algebraically we have
\[
\mathbf{n}(G, \tau_S) = \bigcap_{\chi \in (G, \tau_S)^\wedge } \ker \chi = \bigcap_{x\in s_S (G_d^\wedge)} \ker x = \left[ s_S (G_d^\wedge) \right]^\perp .\eqno \Box
\]

\begin{cor} \label{c11}
{\it Let $G$ be an infinite Abelian group and $S$ be an arbitrary set of sequences in $G$. Then the following statements are equivalent:
\begin{enumerate}
\item[{\rm 1)}] $S\in \mathcal{TBS}(G)$;
\item[{\rm 2)}] $S\in \mathcal{TS}(G)$ and $(G,\tau_S)$ is $MAP$;
\item[{\rm 3)}] $s_S (G_d^\wedge)$ is dense in $G_d^\wedge$.
\end{enumerate} }
\end{cor}

\begin{proof}
$1) \Rightarrow 2)$ is trivial.

$2) \Rightarrow 3)$ By Theorem \ref{t1}, $(G,\tau_S)$ is $MAP$ iff $s_S (G_d^\wedge)$ is dense in $G_d^\wedge$.

$3) \Rightarrow 1)$ For simplicity we set $H:= s_S (G_d^\wedge)$. By  Lemma \ref{l11}, every $\mathbf{u}\in S$ converges to zero in $T_H$. Since $H$ is dense in $G_d^\wedge$, by \cite[Theorem 1.9]{CoR}, $T_H$ is Hausdorff. So $S\in \mathcal{TBS}(G)$.
\end{proof}

{\it Proof of Theorem} \ref{t2}. Set $H:= s_S (G_d^\wedge)$ and $Y=j_S^\wedge \left( (G, \tau_{bS})^\wedge\right)$.
By Theorem \ref{t1}, $H =i_S^\wedge \left( (G, \tau_S)^\wedge\right)$. Since $\tau_S$ is finer than $\tau_{bS}$, the identity map $j: (G,\tau_S) \to (G, \tau_{bS})$ is a continuous isomorphism. Hence $j^\wedge : (G, \tau_{bS})^\wedge \to (G, \tau_{S})^\wedge$ is a continuous monomorphism. Since $j^\wedge_S = i^\wedge_S \circ j^\wedge$, we have $Y\subseteq H$. By  Lemma \ref{l11}, every $\mathbf{u}\in S$ converges to zero in $T_H$. Thus, $\tau_{bS}$ is finer than $T_H$. Hence, by \cite[1.2 and 1.4]{CoR}, $H\subseteq Y$. Thus, $H=Y$.   By \cite[Theorems 1.2 and 1.3]{CoR}, $\tau_{bS} = T_H$.
$\Box$

The following corollary generalizes \cite[Proposition 3.2]{DMT}.

\begin{cor} \label{c12}
{\it Let $G$ be an infinite Abelian group and $S\in \mathcal{TBS}(G)$. Then:
\begin{enumerate}
\item[{\rm 1)}] $\omega (G, \tau_{bS}) = \left| s_S (G_d^\wedge)\right|$;
\item[{\rm 2)}] $\tau_{bS}$ is metrizable iff $s_S (G_d^\wedge)$ is countable.
\end{enumerate} }
\end{cor}

\begin{proof}
1) follows from Theorem \ref{t2} and the property $\omega (G, T_Y)= |Y|$ of the topology $T_Y$ generated by any subgroup $Y\leq G_d^\wedge$.

Put $H:= s_S (G_d^\wedge)$. By Corollary \ref{c11}, $H$ is a point-separating subgroup of $G_d^\wedge$. Thus 2) follows from  Theorem \ref{t2}  and \cite[Theorem 1.11]{CoR}.
\end{proof}

\begin{rmk}
Note that, in general, for $\tau_S$ the equality $\omega (G, \tau_{S}) = \left| s_S (G_d^\wedge)\right|$ is not fulfilled. Indeed, let $S=\{ \mathbf{u}\}$ contain only one $T$-sequence and let $S$ be such that $s_S (G_d^\wedge)$ is countably infinite and dense in $G_d^\wedge$ \cite{DiK}. Then $|s_S (G_d^\wedge)|= \aleph_0$, but $(G, \tau_{S})$ is not metrizable \cite{ZP1}, and hence $\omega (G, \tau_{S})> \aleph_0$. Now, let $(G,\tau)$ be a dense countable subgroup of a locally compact non-compact Abelian metrizable group $X$ with the induced topology and $S=S(G,\tau)$. By \cite[Theorem 1.13]{Ga3}, $\omega (G, \tau_{S})=\omega (G, \tau)= \aleph_0$, but $|s_S (G_d^\wedge)|= | X^\wedge |= \mathfrak{c}$.
\end{rmk}

As usual, the natural homomorphism from an Abelian topological group $G$ into its bidual group $G^{\wedge\wedge}$ is denoted by $\alpha$.
\begin{cor} \label{c13}
{\it Let $G$ be a countably infinite Abelian group and $S\in \mathcal{TBS}(G)$. If $s_S (G_d^\wedge)$  is countable, then $(G, \tau_{bS})^\wedge $ is discrete. In particular, $(G, \tau_{bS})$ is not reflexive.}
\end{cor}

\begin{proof}
By Corollary \ref{c12}, $\tau_{bS}$ is metrizable. Hence the completion $\overline{G}$ of $(G, \tau_{bS})$ is a metrizable compact group. Thus, $\overline{G}$ is determined \cite{Aus, Cha}. Hence $(G, \tau_{bS})^\wedge$ is topologically isomorphic to the discrete group $\overline{G}^\wedge$. In particular, $(G, \tau_{bS})^{\wedge\wedge}$ is an infinite compact group. Since $G$ is countable, $\alpha (G) \not= (G, \tau_{bS})^{\wedge\wedge}$. So $(G, \tau_{bS})$ is not reflexive.
\end{proof}

\begin{cor} \label{c14}
{\it Let $G$ be a countably infinite Abelian group, $\mathbf{u}$ be a $TB$-sequence and
$j : (G, \mathbf{u}) \to (G, b\mathbf{u})$ be the identity continuous isomorphism. If  $s_\mathbf{u} (G_d^\wedge)$ is countable, then $j^\wedge$ is a topological isomorphism.}
\end{cor}

\begin{proof}
$j^\wedge$ is a continuous isomorphism by  Corollary \ref{c01}.  Since, by   Theorem \ref{t01}, $(G, \mathbf{u})^\wedge $ is complete and countable, it is discrete. Thus $(G, \tau_{bS})^\wedge$ is also discrete and $j^\wedge$ is a topological isomorphism. $\Box$
\end{proof}

%\newdefinition{rmk}{Remark}
%\theoremstyle{definition}
%Theorem \ref{t3} immediately follows from the following:

{\it Proof of Theorem} \ref{t12}.
(i) follows from  Theorem \ref{t1}(1) and the definitions of $\mathfrak{g}$-closed subgroups.

(ii) It is clear that  $H$ is a $\mathfrak{g}$-closed dense subgroup of ${\rm cl} H$.  Put
\[
S:= \{ \mathbf{u}\in \left( ({\rm cl} H)^\wedge \right)^\mathbb{N} : \; H \leq s_\mathbf{u} ({\rm cl} H)\}.
\]
By   the definition of $\mathfrak{g}$-closed subgroups, $H=s_S ({\rm cl} H)$.  Since $H$ is dense in ${\rm cl} H$, $S\in \mathcal{TBS}(({\rm cl} H)^\wedge)$   by Corollary \ref{c11}. Now the assertion follows from Theorem \ref{t1}(1).
$\Box$

Let $G$ be a $MAP$ Abelian topological  group, $X=\widehat{G}$ and $\alpha$ be the natural including of $G$ into $G^{\wedge\wedge}$. Since $G$ is $MAP$, $\alpha$ is injective. The weak and weak$^\ast$ group topologies on $X$ we denote by $\tau_w$ and $\tau_{w^\ast}$ respectively, i.e.,  $\tau_w =\sigma(X,G)$ and $\tau_{w^\ast} =\sigma(X,G^{\wedge\wedge})$. The compact-open topology on $X$ is denoted by $\tau_{co}$. Then $\tau_w \subseteq \tau_{w^\ast} \subseteq \tau_{co}$.  Let $t: X_d \to (X, \tau_{co}) (=G^\wedge), t(x)=x,$ be the natural continuous isomorphism and $\mathfrak{b}:= t^\wedge$ be its conjugate continuous monomorphism. Set $bG:= X_d^\wedge$. It is well known that the group $bG$ with the continuous monomorphism $\mathfrak{b}\circ\alpha$ is the Bohr compactification of $G$ (although $\alpha$ need not be continuous, but $\mathfrak{b}\circ\alpha$ is always continuous since $(\mathfrak{b}\circ\alpha (g),x)=(\alpha(g),t(x))= (x,g)$ for every $g\in G$ and $x\in X_d$). We shall algebraically identify $G$ and  $G^{\wedge\wedge}$ with their images $\mathfrak{b}\circ\alpha (G)$ and $\mathfrak{b}\left( G^{\wedge\wedge} \right)$ respectively saying that they are subgroups of $bG$.
It is clear that
\begin{equation} \label{21}
\mathfrak{g}_{bG} (\mathfrak{b}\circ\alpha (G)) \subseteq \mathfrak{g}_{bG} (\mathfrak{b}\left( G^{\wedge\wedge})\right).
\end{equation}
\begin{pro} \label{p12}
Let $G$ be a $MAP$ Abelian topological  group and $X=\widehat{G}$. The following statements are equivalent:
\begin{enumerate}
\item[{\rm (i)}] $\mathbf{s}(X,\tau_{w}) = \mathbf{s}(X, \tau_{w^\ast})$;
\item[{\rm (ii)}] $\mathfrak{g}_{bG} (\mathfrak{b}\circ\alpha (G)) = \mathfrak{g}_{bG} (\mathfrak{b}\left( G^{\wedge\wedge})\right)$.
\end{enumerate}
In particular, if $G$ is reflexive, then (i) and (ii) are fulfilled.
\end{pro}

\begin{proof}
(i)$\Rightarrow$(ii). By (\ref{21}), we have to show that $\mathfrak{g}_{bG} (\mathfrak{b}\circ\alpha (G)) \supseteq \mathfrak{g}_{bG} (\mathfrak{b}\left( G^{\wedge\wedge})\right).$ Let $\mathbf{u} =\{ u_n\}_{n\in\omega} \subset X$ be such that $\mathfrak{b}\circ\alpha (G) \subseteq s_\mathbf{u} (bG)$. This means that $(\mathfrak{b}\circ\alpha (g), u_n) = (u_n, g)\to 1$ for every $g\in G$, i.e., $\mathbf{u}\in S(X,\tau_{w})$. By hypothesis, $\mathbf{u}\in S(X, \tau_{w^\ast})$ either. Hence
\[
(\mathfrak{b}(\chi), u_n) = (\chi, u_n) \to 1 \mbox{ for every } \chi\in G^{\wedge\wedge},
\]
i.e., $\mathfrak{b}(G^{\wedge\wedge}) \subseteq s_\mathbf{u} (bG)$. So $\mathfrak{g}_{bG} (\mathfrak{b}\circ\alpha (G)) \supseteq \mathfrak{g}_{bG} (\mathfrak{b}\left( G^{\wedge\wedge})\right).$

(ii)$\Rightarrow$(i). Since $\tau_w \subseteq \tau_{w^\ast} $, we have to show only that if $\mathbf{u}=\{ u_n\}_{n\in\omega} \in S(X,\tau_{w})$, then also $\mathbf{u}\in S(X, \tau_{w^\ast})$. Assuming the converse we can find $\chi \in G^{\wedge\wedge}$ such that
\[
(\chi, u_n) \not\to 1, \mbox{ at } n\to\infty.
\]
Then $\mathfrak{b}(\chi) \not\in s_\mathbf{u} (bG)$. Thus $\mathfrak{b}(\chi) \not\in \mathfrak{g}_{bG} (\mathfrak{b}\circ\alpha (G))$. A contradiction.
\end{proof}

{\it Proof of Proposition} \ref{p01}.
Since $G$ is reflexive, $\tau_w = \tau_{w^\ast} $. By Proposition \ref{p12}, we have $\mathfrak{g}_{bG} (\mathfrak{b}\circ\alpha (G)) = \mathfrak{g}_{bG} (\mathfrak{b}\left( G^{\wedge\wedge})\right)$. By \cite[Theorem 2.4]{ChMPT}, the dual group of a separable metrizable Abelian group $G$ is  sequential. By \cite[Theorem 1.13]{Ga3}, $G^\wedge$ is a $s$-group. Hence, by Theorem \ref{t12}, $\mathfrak{g}_{bG} (\mathfrak{b}\circ\alpha (G)) = \mathfrak{b}\circ\alpha (G)$ and $\mathfrak{b}\circ\alpha (G)$ is a $\mathfrak{g}$-closed subgroup of $bG$.
$\Box$

Now we discuss the minimality of $|S|$ of $TS$-sets $S$ which generate the same topology.
\begin{defin}
{\it Let $G$ be an Abelian group.
\begin{enumerate}
\item[{\rm (1)}] If $S\in \mathcal{TS} (G)$, we put
\[
r_s (S)= \inf\left\{ |B| : \,  (G, \tau_B) \cong (G, \tau_S) \mbox{ and } B \in \mathcal{TS} (G) \right\},
\]
\[
r_s^\wedge (S)= \inf\left\{ |B| : \, s_B (G_d^\wedge) =s_S (G_d^\wedge) \mbox{ and } B \in \mathcal{TS} (G) \right\}.
\]
\item[{\rm (2)}] If $S\in \mathcal{TBS} (G)$, we put
\[
r_b (S)= \inf\left\{ |B| : \,  (G, \tau_{bB}) \cong (G, \tau_{bS}) \mbox{ and } B \in \mathcal{TBS} (G) \right\},
\]
\[
r_b^\wedge (S)= \inf\left\{ |B| : \, s_B (G_d^\wedge) =s_S (G_d^\wedge) \mbox{ and } B \in \mathcal{TBS} (G) \right\}.
\]
\end{enumerate} }
\end{defin}
\begin{rmk}
Let $(G,\tau)$ be a $s$-group and $\tau =\tau_S$ for some $S\in\mathcal{TS}(G)$. Then  the number $ r_s (S)$ coincides with the number $r_s (G,\tau)$ that is defined in \cite{Ga3}.
\end{rmk}
\begin{pro} \label{p11}
{\it Let $G$ be an infinite Abelian group.
\begin{enumerate}
\item[{\rm 1)}] If $S\in \mathcal{TS} (G)$, then $r_s^\wedge (S) \leq r_s (S)$.
\item[{\rm 2)}] If $S\in \mathcal{TBS} (G)$, then $r_s^\wedge (S) =r_b^\wedge (S)= r_b (S)$.
\item[{\rm 3)}] If $S\in \mathcal{TS} (G)$ is finite, then $r_s (S)=r_s^\wedge (S)=1$.
\end{enumerate} }
\end{pro}

\begin{proof}
1) Let $B \in \mathcal{TS} (G)$ be such that $(G, \tau_B) \cong (G, \tau_S)$. By Theorem \ref{t1}(1), algebraically,
\[
s_B (G_d^\wedge) =(G, \tau_B)^\wedge =(G, \tau_S)^\wedge =s_S (G_d^\wedge).
\]
So $| B| \geq r_s^\wedge  (S)$. Thus $r_s^\wedge (S) \leq r_s (S)$.

2) Let $S\in \mathcal{TBS} (G)$. By Corollary \ref{c11}, $s_S (G_d^\wedge)$ is dense in $G_d^\wedge$. Hence, if $s_B (G_d^\wedge)=s_S (G_d^\wedge)$ for $B\in \mathcal{TS} (G)$, then, by Corollary \ref{c11}, $B\in \mathcal{TBS} (G)$. Thus, $r_s^\wedge (S) =r_b^\wedge (S)$.

Let $B\in \mathcal{TBS} (G)$. By Theorem \ref{t2} and  \cite[Theorem 1.3]{CoR}, $s_B (G_d^\wedge)=s_S (G_d^\wedge)$ if and only if $\tau_{bB} =\tau_{bS}$. So $r_b (S) =r_b^\wedge (S)$.

3)  By Proposition 2.6 of \cite{Ga3}, $r_s (S)=1$ and the assertion follows from item 1).
\end{proof}

\begin{exa}
Let $(G,\tau)$ be a dense countably infinite subgroup of a compact infinite metrizable Abelian group with the induced topology. Thus $(G,\tau)$ is  a $s$-group. Set $S=S(G,\tau)$. Since $(G, \mathbf{u})$ is either discrete or non metrizable, by Proposition \ref{p11}(3), we have $r_s (S)\geq \aleph_0$. On the other hand, by Theorem \ref{t1}, algebraically, $s_{S}(G_d^\wedge) = (G,\tau)^\wedge$ is  a countable subgroup of $G_d^\wedge$. So, by \cite{DiK}, there exists a $TB$-sequence $\mathbf{u}$ in $G$ such that $s_{\mathbf{u}}(G_d^\wedge) =s_{S}(G_d^\wedge)$. Thus $r_s^\wedge (S)=1$ and hence $r_s (S) > r_s^\wedge (S)$. We do not know any characterization of those $s$-groups for which $r_s (S) = r_s^\wedge (S)$. $\Box$
\end{exa}

\section{Structure of Abelian $s$-groups} \label{sec2}

In the following proposition we describe all sequences converging to zero in $(G,\mathbf{u})$.
\begin{pro} \label{p31}
{\it Let $\mathbf{u}=\{ u_n \}$ be a $T$-sequence in an Abelian group $G$. A sequence $\mathbf{v}=\{ v_n \}$ converges to zero in $(G,\mathbf{u})$ if and only if there are  $m\geq 0$ and $n_0 \geq 0$ such that for every $n\geq n_0$ each member $v_n \not= 0$ can be represented in the form
\[
v_n = a^n_{1} u_{k_1^n} +\dots + a^n_{l_n} u_{k_{l_n}^n},
\]
where $k_1^n <\dots <k_{l_n}^n$, $|a^n_{1}| + \dots +|a^n_{l_n}| \leq m+1$ and  $k_1^n \to \infty$.}
\end{pro}

\begin{proof}
If either $\mathbf{u}$ or $\mathbf{v}$ is trivial, the proposition is evident. Assume that $\mathbf{u}$ and $\mathbf{v}$ are non-trivial.
The sufficiency is clear. Let us prove the necessity.
Since the subgroup $\langle\mathbf{u}\rangle$ of $G$ is open in $\tau_\mathbf{u}$, there is $n_0$ such that $v_n \in \langle\mathbf{u}\rangle$ for every $n\geq n_0$. Thus, without loss of generality, we may assume that $\langle\mathbf{u}\rangle =G$. Since $\mathbf{v}\cup \{ 0\}$ is compact and $\langle\mathbf{u}\rangle =G$, by \cite[Theorem 2.10]{Ga3}, there is $m\geq 0$ such that $\mathbf{v}\subset A(m,0)$. So, if $v_n \not= 0$, then
\begin{equation} \label{e1}
v_n = a^n_{1} u_{k_1^n} +\dots + a^n_{l_n} u_{k_{l_n}^n}, \mbox{ where } k_1^n <\dots <k_{l_n}^n \mbox{ and } |a^n_{1}| + \dots +|a^n_{l_n}| \leq m+1.
\end{equation}
Also we may assume that for any fix $n$ every sum of terms of the form $a^n_{i} u_{k_i^n}$ in (\ref{e1}) is not equal to zero (in particular, $a^n_{i} u_{k_i^n}\not= 0$ for $i=1,\dots, l_n$).

Now we have to show that $k_1^n \to \infty$. Assuming the converse  and passing to a subsequence we may suppose that $k_1^n = k_1$, $a^n_{1} =a_1$ and $a^n_{k_1^n} u_{k_1^n}=a_1 u_{k_1}\not= 0$ for every $n$. So
\[
v_n = a_1 u_{k_1} +a^n_{2} u_{k_2^n} +\dots + a^n_{l_n} u_{k_{l_n}^n} = a_1 u_{k_1}+ w_n^1.
\]
If $k_2^n \to\infty$, we obtain that $w_n^1$ converges to zero. Hence $0\not= a_1 u_{k_1}= v_n - w_n^1 \to 0$. This is a contradiction. Thus, there is a bounded subsequence of $\{ k_2^n \}$. Passing to a subsequence  we may suppose that $k_2^n = k_2$, $a^n_{2} =a_2$ and $a^n_{2} u_{k_2^n}=a_2 u_{k_2}\not= 0$ for every $n$. So
\[
v_n = a_1 u_{k_1}+a_2 u_{k_2} +a^n_{3} u_{k_3^n} +\dots + a^n_{l_n} u_{k_{l_n}^n} = a_1 u_{k_1}+ a_2 u_{k_2}+ w_n^2.
\]
By hypothesis, $a_1 u_{k_1}+ a_2 u_{k_2} \not= 0$. And so on. Since
\[
0< |a_1 | <  |a_1 |+ |a_2 |< \dots \leq m+1,
\]
after at most $m+1$ steps, we obtain that there is a {\it fix } and {\it non-zero} subsequence of $\mathbf{v}$. Thus $v_n \not\to 0$. This contradiction shows that $k_1^n \to \infty$.
\end{proof}

The following theorem makes more precise Theorem 2.9 of \cite{Ga3}.
\begin{theorem} \label{t02}
Let $\mathbf{u}=\{ u_n \}$ be a $T$-sequence in an Abelian group $G$ such that $\langle\mathbf{u}\rangle =G$. Then $(G,\mathbf{u})$ is a quotient  group of  $(\mathbb{Z}^\mathbb{N}_0, \mathbf{e})$ under the sequence-covering homomorphism
\[
 \pi ((n_1, n_2,\dots, n_m, 0,\dots))= n_1 u_1 +n_2 u_2 +\dots +n_m u_m .
\]
\end{theorem}

\begin{proof}
Taking into account  Theorem 2.9 of \cite{Ga3}, we have to show only  that $\pi$ is sequence-covering. Let $\mathbf{v} =\{ v_n \} \in S(G,\mathbf{u})$. By Proposition \ref{p31}, for some natural number $m$ we can represent every $v_n \not= 0$ in the form
\[
v_n = a^n_{1} u_{k_1^n} +\dots + a^n_{l_n} u_{k_{l_n}^n},
\]
where $k_1^n <\dots <k_{l_n}^n$, $|a^n_{1}| + \dots +|a^n_{k_{l_n}^n}| \leq m+1$ and $k_1^n \to \infty$.
 Set
\[
e'_n = a^n_{1} e_{k_1^n} +\dots + a^n_{l_n} e_{k_{l_n}^n} \mbox{ if } v_n \not= 0, \mbox{ and } e'_n =0 \mbox{ if } v_n =0.
\]
Then $e'_n \to 0$ in $(\mathbb{Z}^\mathbb{N}_0, \mathbf{e})$ and $\pi (e'_n) = v_n$.
\end{proof}

Let $\{ (G_i, \tau_i) \}_{i\in I}$, where $I$ is a non-empty set of indices, be a family of Hausdorff  topological groups. For every $i\in I$ fix $U_i \in \mathcal{U}_{G_i}$ and put
\[
 \sum_{i\in I} U_i :=\left\{ (g_i)_{i\in I} \in \sum_{i\in I} G_i : \; g_i \in U_i \mbox{ for  all } i \in I \right\}.
\]
Then the sets of the form $ \sum_{i\in I} U_i $, where $U_i \in \mathcal{U}_{G_i}$ for every $i\in I$, form a neighborhood basis at the unit of a Hausdorff group topology $\tau^r$ on $\sum_{i\in I} G_i$ that is called the rectangular (or box) topology.

Let $\mathbf{u}=\{ g_n \}$ be an arbitrary sequence in $S(G_i, \tau_i)$. Evidently, the sequence $j_i (\mathbf{u})$ converges to zero in $\tau^r$. Thus,  the set
$\bigcup_{i\in I} j_i \left( S(G_i, \tau_i)\right)$
is a $TS$-set of sequences in $\sum_{i\in I} G_i$. So, if $(G_i, \tau_i)$ is a $s$-group for all $i\in I$, then Definition \ref{d04} is correct. Moreover, we can prove the following:
\begin{pro}
{\it Let $G= \sum_{i\in I} G_i$, where $(G_i, \tau_i)$ is a $s$-group for every $i\in I$. Set $S:=\bigcup_{i\in I} j_i \left( S(G_i, \tau_i)\right)$. The topology $\tau_S$ on $G$ coincides with the finest Hausdorff group topology $\tau'$ on $G$ for which all inclusions $j_i$ are continuous.}
\end{pro}

\begin{proof}
Fix $i\in I$. By construction, for every $\{ u_n\} \in S(G_i, \tau_i)$, $j_i (u_n) \to e_G$ in $\tau_S$. By \cite[Theorem 2.4]{Ga3}, the inclusion $j_i$ is continuous. Thus, $\tau_S \subseteq \tau'$. Conversely, if $j_i$ is continuous, then $j_i \left( S(G_i, \tau_i)\right) \subset S(G,\tau')$. Hence $S\subseteq S(G,\tau')$ and $\tau' \subseteq \tau_S$ by the definition of $\tau_S$.
\end{proof}

\begin{theorem} \label{t41}
Let $(X,\tau)$ be an Abelian $s$-group. Set $I=S(X,\tau)$. For every $\mathbf{u}\in I$, let $p_\mathbf{u} (\langle\mathbf{u}\rangle , \mathbf{u})\to X, p_\mathbf{u} (g)=g$, be the natural including of $(\langle\mathbf{u}\rangle , \mathbf{u})$ into $X$. Then the natural homomorphism
\[
p: {s-\sum}_{\mathbf{u}\in S(X,\tau)} (\langle \mathbf{u}\rangle, \mathbf{u}) \to X, \; p\left( (x_\mathbf{u}) \right) = \sum_\mathbf{u} p_\mathbf{u} (x_\mathbf{u}) =\sum_\mathbf{u} x_\mathbf{u},
\]
 is a quotient sequence-covering map.
\end{theorem}

\begin{proof}
 Set
\[
G:={s-\sum}_{\mathbf{u}\in S(X,\tau)} (\langle \mathbf{u}\rangle, \mathbf{u})  \mbox{ and }  S:= \bigcup_{\mathbf{u}\in S(X,\tau)} j_\mathbf{u} (S(\langle \mathbf{u}\rangle, \mathbf{u})) \in \mathcal{TS} (G).
\]
Since any element of $X$  can be regarded as the first element of some sequence $\mathbf{u}\in S(X,\tau)$, $p$ is surjective. By construction, $p$ is sequence-covering.

Let $\mathbf{v}=\{ v_n\} \in S$. By construction, $p(v_n) = v_n \to 0$ in $\tau$. Thus, by  \cite[Theorem 2.4]{Ga3}, $p$ is continuous. Set $H= \ker p$. By  \cite[Theorem 1.11]{Ga3}, $G/H \cong (X, \tau_{p(S)} )$. Since, by construction, $p(S)=S(X,\tau)$, we obtain that $G/H \cong (X,\tau)$ by Proposition 1.7 of \cite{Ga3}.
\end{proof}

To prove Theorem \ref{t06} we need the following proposition.
\begin{pro} \label{p42}
Let $\{ (X_i, \nu_i)\}_{i\in I}$ and $\{ (G_i, \tau_i)\}_{i\in I}$ be non-empty families of Abelian $s$-groups and let $\pi_i : G_i \to X_i$ be a quotient sequence-covering map for every $i\in I$. Set $X= {s-\sum}_{i\in I} X_i$, $G={s-\sum}_{i\in I} G_i$ and $\pi : G\to X, \pi ((g_i))=(\pi_i(g_i))$. Then $\pi$ is a quotient  map.
\end{pro}

\begin{proof}
It is clear that $\pi$ is surjective. Set
\[
S_X :=  \bigcup_{i\in I} j_i ( S(X_i, \nu_i)) \mbox{ and } S_G :=  \bigcup_{i\in I} j_i (S(G_i, \tau_i)).
\]
Since $\pi_i$ is sequence-covering, we have $\pi_i (S(G_i, \tau_i))=S(X_i, \nu_i)$. Hence $\pi (S_G) = S_X$. Thus, by  \cite[Theorem 2.4]{Ga3}, $\pi$ is continuous. By  \cite[Theorem 1.11]{Ga3}, $G/\ker (\pi) \cong (X, \tau_{\pi (S_G)})$. Hence $G/\ker (\pi) \cong X$ and $\pi$ is a quotient map.
\end{proof}

{\it Proof of Theorem} \ref{t06}.
Let $I=S(X,\tau)$.  For every $\mathbf{u} \in I$, put $G_\mathbf{u} = (\mathbb{Z}^\mathbb{N}_0, \mathbf{e})$, $X_\mathbf{u} =(\langle\mathbf{u}\rangle , \mathbf{u}) $ and
$\pi_\mathbf{u} ((n_1, \dots, n_m, 0,\dots))= n_1 u_1 +\dots +n_m u_m$. Let $p_\mathbf{u} (\langle\mathbf{u}\rangle , \mathbf{u})\to X, p_\mathbf{u} (g)=g$, be the natural including of $(\langle\mathbf{u}\rangle , \mathbf{u})$ into $X$. Then the theorem immediately follows from Theorems \ref{t02} and \ref{t41} and Proposition \ref{p42}.
$\Box$

 The following theorem is a natural counterpart of \cite[Theorem 4.1]{Siw}:
\begin{theorem} \label{t07}
Let $(X,\tau)$ be a non-trivial Hausdorff Abelian topological group. The following statements are equivalent:
\begin{enumerate}
\item[{\rm (i)}] $(X,\tau)$ is a $s$-group;
\item[{\rm (ii)}] every continuous sequence-covering homomorphism from an Abelian $s$-group onto $(X,\tau)$ is quotient.
\end{enumerate}
\end{theorem}

\begin{proof}
 (i) $\Rightarrow$ (ii). Let $p: G\to X$ be a sequence-covering continuous homomorphism from a $s$-group $(G, \nu)$ onto $X$. Set $H= \ker p$. We have to show that $p$ is quotient, i.e., $X\cong G/H$. Since $p$ is surjective, by  \cite[Theorem 1.11]{Ga3}, we have $G/H \cong (X, \tau_{p(S(G,\nu ))})$. By hypothesis and Proposition 1.7 of \cite{Ga3}, $p(S(G,\nu ))=S(X,\tau)$ and $\tau= \tau_{S(X,\tau)}$. Thus $G/H \cong X$.

(ii) $\Rightarrow$ (i). Let $I=S(X,\tau)$, $S:= \bigcup_{\mathbf{u}\in S(X,\tau)} j_\mathbf{u} (S(\langle \mathbf{u}\rangle, \mathbf{u})) \in \mathcal{TS} (G),$ $G:={s-\sum}_{\mathbf{u}\in S(X,\tau)} (\langle \mathbf{u}\rangle, \mathbf{u})$   and
\[
p: G \to X, \; p\left( (x_\mathbf{u}) \right) = \sum_\mathbf{u} p_\mathbf{u} (x_\mathbf{u}) =\sum_\mathbf{u} x_\mathbf{u},
\]
Since every $(\langle \mathbf{u}\rangle, \mathbf{u})$ is a $s$-group, $G$ is a $s$-group either. By  \cite[Theorem 2.4]{Ga3}, $p$ is continuous. Since  $p$ is sequence-covering, by hypothesis, $p$ is quotient. Thus $(X,\tau) \cong G/\ker p$. By Theorem  \cite[Theorem 1.11]{Ga3}, we also have $G/\ker p \cong (X, \tau_{p(S)})$. Thus $\tau = \tau_{p(S)}$ and $(X,\tau)$ is a $s$-group.
\end{proof}

\section{Countable $s$-sums of $s$-groups} \label{sec3}

We start from the description of the topology $\tau_S$ on $G$ for countably infinite $S\in \mathcal{TS}(G)$.

\begin{pro} \label{p20}
{\it Let  $S=\{ \mathbf{u}_n\}_{n\in\omega} \in \mathcal{TS}(G)$.  Then the family $\mathcal{U}$ of all the sets of the form
\[
\sum_n W_n = \bigcup_{n=0}^\infty (W_0 +W_1 +\dots + W_n ), \mbox{  where } 0\in W_n \in \tau_{\mathbf{u}_n}, \; n\geq 0,
\]
forms an open basis at $0$ of $\tau_S$.}
\end{pro}

Proposition \ref{p20} is an immediate corollary of the following two assertions.
\begin{lemma} \label{l21}
{\it Let $S\in \mathcal{TS}(G)$ for an Abelian group $G$ and $S=\cup_{i\in I} S_i$, where $I$ is a non-empty set of indices. Then $\tau_S \subseteq \bigwedge_i \tau_{S_i}$.}
\end{lemma}
\begin{proof}
It is clear that $S_i \in \mathcal{TS}(G)$ and $\tau_{S} \subseteq \tau_{S_i}$ for every $i\in I$. Thus, if $U\in \tau_{S}$, then $U\in \tau_{S_i}$ for every $i\in I$. Hence $\tau_S \subseteq \bigwedge_i \tau_{S_i}$.
\end{proof}

\begin{pro} \label{p21}
{\it Let  $S=\cup_{n=0}^\infty S_n \in \mathcal{TS}(G)$.  Then the family $\mathcal{U}$ of all the sets of the form
\[
\sum_n W_n := \bigcup_{n=0}^\infty (W_0 +W_1 +\dots + W_n ), \mbox{  where } 0\in W_n \in \tau_{S_n}, \; n\geq 1,
\]
is an open basis at $0$ of $\tau_S$.}
\end{pro}

\begin{proof} It is clear that $S_n \in \mathcal{TS}(G)$ for every $n$.

1. We claim that $\mathcal{U}$ {\it forms an open basis at zero of a Hausdorff group topology $\tau$ on $G$}. For this we have to check the five conditions of Theorem 4.5 of \cite{HR1}.

Let $\sum_n W_n \in \mathcal{U}$. To prove (i) choose $V_n \in \tau_{S_n}$ such that $V_n + V_n \subseteq W_n$. Then
\[
\sum_n V_n +\sum_n V_n \subseteq \sum_n (V_n +V_n) \subseteq \sum_n W_n.
\]

(ii) and (iv) are trivial.

To prove (iii) let $g=w_{n_1} +\dots + w_{n_m} \in \sum_n W_n$, where $w_{n_k}\not= 0, k=1,\dots, m$. If $n\not\in \{ n_1,\dots, n_m\}$, we set $V_n = W_n$. If $n=n_k$, we may choose an open neighborhood $V_{n_k}$ of zero in $\tau_{S_{n_k}}$ such that $w_{n_k} +V_{n_k} \subset W_{n_k}$. Then $g+\sum_n V_n \subseteq \sum_n W_n$.

To prove (v) let $\sum_n V_n, \sum_n W_n \in \mathcal{U}$. Set $F_n =W_n \cap V_n \in \tau_{S_{n}}$. Then $\sum_n F_n \subseteq \sum_n V_n \cap \sum_n W_n$.

2. We claim that $\tau \subseteq \tau_S$. By the definition of $\tau_S$ we have to show only that every $\mathbf{u} =\{ u_k \} \in S_n$ converges to zero in $\tau$. Let $\sum_n W_n \in \mathcal{U}$. Since $W_n \in \tau_{S_n}$, $u_k \in W_n \subset \sum_n W_n$ for all sufficiently large $k$. Thus
$\mathbf{u} $ converges to zero in $\tau$.

3. We claim that $\tau =\tau_S$. Let $U \in \tau_S$ be an arbitrary neighborhood of zero. Then there is a sequence of open neighborhoods of zero $U_n \in \tau_S, n\geq 0,$ such that $U_0 +U_0 \subseteq U$ and $U_n +U_n \subseteq U_{n-1}, n\geq 1$. By Lemma \ref{l21}, $\tau_S \subseteq \bigwedge_n \tau_{S_n}$. Hence for every $n\geq 0$ we may choose an open neighborhood $W_n$ of zero in $\tau_{S_n}$ such that $W_n \subset U_n$. It is clear that $\sum_{n} W_n \subseteq U$.
\end{proof}

To prove Theorem \ref{t4}, we need the following proposition.
\begin{pro} \label{p21}
{\it Let $\{ G_n \}_{n\in \omega}$ be a sequence of Abelian groups and let $\mathbf{u}_n$ be a $T$-sequence in $G_n$ for every $n\in \omega$. Set $G=\sum_{n\in \omega} G_n$ and $S=\{ j_n (\mathbf{u}_n) \}_{n\in \omega}$. Then $(G,\tau_S)$ is a complete sequential group, $\tau_S =\tau^r$ and
\[
(G,\tau_S)^\wedge = \prod_{n\in \omega} (G_n, \mathbf{u}_n)^\wedge .
\]
Moreover, if all $G_n$ are countably infinite, then $(G,\tau_S)^\wedge $ is a Polish group.}
\end{pro}

\begin{proof}
$(G,\tau_S)$ is a complete sequential group by Theorem 2.7 of \cite{Ga3}. By Proposition \ref{p20},  $\tau_S =\tau^r$.  Thus, by \cite{Kap}, $(G,\tau_S)^\wedge = \prod_{n\in \omega} (G_n, \mathbf{u}_n)^\wedge$. If all $G_n$ are countably infinite, then, by  Theorem \ref{t01}, all $(G_n, \mathbf{u}_n)^\wedge$ are Polish. Hence $(G,\tau_S)^\wedge $ is a Polish group either.
\end{proof}

{\it Proof of Theorem} \ref{t4}.
Set $G' =\sum_{n\in \omega} G_n$, where $G_n =G$ for every $n\in\omega$, and $S' =\{ j_n (\mathbf{u}_n) \}_{n\in \omega}$. Then, by Proposition \ref{p21},
\[
(G',\tau_{S'})^\wedge = \prod_{n\in \omega} (G, \mathbf{u}_n)^\wedge
\]
 is a Polish group.

Set $p: (G',\tau_{S'}) \to (G,\tau_S), p((g_n)) = \sum_n g_n$. Since $p(j_n (\mathbf{u}_n) ) =\mathbf{u}_n$ converges to zero in $(G,\tau_S)$, $p$ is continuous by Theorem 2.4 of \cite{Ga3}. Set $H=\ker p$. Since $p(S')=S$, by \cite[Theorem 1.11]{Ga3}, $(G,\tau_S)\cong (G',\tau_{S'})/H$. Then the conjugate homomorphism $p^\wedge$ is a continuous isomorphism from $(G,\tau_S)^\wedge$ onto the annihilator $H^\perp$ of $H$ in $(G',\tau_{S'})^\wedge $.
By \cite[Theorem 2.7]{Ga3}, every compact subset of $(G,\tau_S)$ is contained in a compact subset $K_n$ of the form
\[
K_n := \left[ \bigcup_{i=0}^n \left( \mathbf{u}_i \cup (-\mathbf{u}_i) \right) \right] + \cdots + \left[ \bigcup_{i=0}^n \left( \mathbf{u}_i \cup (-\mathbf{u}_i) \right) \right]
\]
with $n+1$ summands. It is clear that a subset $K'_n$ of $G'$ of the form
\[
K'_n := \left[ \bigcup_{i=0}^n \left( j_i (\mathbf{u}_i) \cup (-j_i (\mathbf{u}_i)) \right) \right] + \cdots + \left[ \bigcup_{i=0}^n \left( j_i (\mathbf{u}_i) \cup (-j_i (\mathbf{u}_i)) \right) \right]
\]
with $n+1$ summands, is compact. Since  $p(K'_n )=K_n$ and $p$ is onto and continuous, $p$ is  compact-covering.  Thus, by \cite[Lemma 5.17]{Aus},  $p^\wedge$ is an embedding of $(G,\tau_S)^\wedge$ into the Polish group $(G',\tau_{S'})^\wedge $. So $(G,\tau_S)^\wedge \cong H^\perp$ is a Polish group.
$\Box$

{\it Proof of Proposition} \ref{p2}. For every $X_n$ there is a countably infinite Abelian group $G_n$ and a $TB$-sequence $\mathbf{u}_n$ in $G_n$ such that $(G_n , \mathbf{u}_n)^\wedge \cong X_n$ \cite{Ga2}. Set $G=\sum_{n\in \omega} G_n$ and $S=\{ j_n(\mathbf{u}_n) \}_{n\in \omega}$. Then the proposition follows from Proposition \ref{p21}.
$\Box$

{\it Proof of Proposition} \ref{p3}. By Proposition 2.9 of \cite{Ga2}, there is a $TB$-sequence $\mathbf{u}$ on $\mathbb{Z}^2$, such that $(\mathbb{Z}^2, \mathbf{u})^\wedge \cong \mathbb{R}$. Since
$\sum_{n\in \omega} \mathbb{Z}^2 \cong \mathbb{Z}^\mathbb{N}_0$, the assertion follows from Proposition \ref{p21}.
$\Box$

\section{Open questions}

We start from a question concerning Theorem \ref{t12}:
\begin{problem}
{\it Let $X$ be a compact Abelian group and $H$ be a $\mathfrak{g}$-closed non-dense subgroup of $X$. Is there $S\in \mathcal{TS}(\widehat{X})$ such that $i^\wedge_S \left( (\widehat{X}, \tau_S )^\wedge \right)=H$}?
\end{problem}

As it was noted, if $G$ is a {\it  separable} metrizable Abelian topological group, then, by \cite[Theorem 1.7]{ChMPT}, the dual group $G^\wedge$ is sequential. Hence  $G^\wedge$ is a $s$-group. The following questions are open:

\begin{problem}
{\it Let $G$ be a non-separable metrizable (resp. Fr\'{e}chet-Urysohn or sequential) Abelian group. Is $G^\wedge$ a $s$-group}?
\end{problem}
\begin{problem}
{\it Let $G$ be an  Abelian $s$-group. When $G^\wedge$ is a $s$-group}?
\end{problem}
\begin{problem}
{\it Let $G$ be an  Abelian (resp. metrizable,  Fr\'{e}chet-Urysohn, sequential or a $s$-group) topological group such that $G^\wedge$ is a (resp. metrizable,  Fr\'{e}chet-Urysohn or sequential) $s$-group. What we can say additionally about $G$ and  $G^\wedge$}?
\end{problem}
For example, if $G$ is metrizable and $G^\wedge$ is Fr\'{e}chet-Urysohn, then, by \cite[Theorem 2.2]{ChMPT}, $G^\wedge$ is locally compact metrizable group.

Let $G$ be an Abelian group and $S\in \mathcal{TBS}(G)$. Theorem \ref{t2} gives a complete description of the topology $\tau_{bS}$ on  $G$. On the other hand, we do not know any description of the topology on the dual group.
\begin{problem}
{\it Describe the topology of $(G, \tau_{bS})^\wedge$.}
\end{problem}
By Corollary \ref{c01}, $(G, \tau_{bS})^\wedge =(G, \tau_{S})^\wedge$ algebraically. It is natural to ask:
\begin{problem}
{\it  When the groups $(G, \tau_{bS})^\wedge$ and $(G, \tau_{S})^\wedge$ are topologically isomorphic}? {\it In particular, when $(G, \tau_{\mathbf{u}})^\wedge \cong (G, \tau_{b\mathbf{u}})^\wedge$}?
\end{problem}
Let $G$ be a countably infinite Abelian group and $S\in \mathcal{TBS}(G)$. By Corollary \ref{c13}, if $(G, \tau_{bS})^\wedge$ is countable, then $(G, \tau_{bS})$ is not reflexive.
\begin{problem}
{\it Is there a $S\in \mathcal{TBS}(G)$ for a {\rm  countably} infinite Abelian group $G$ such that  $(G, \tau_{bS})$ is reflexive}?
\end{problem}
Note that the positive answer to the last question will give the positive answer to the  following general problem:
\begin{problem}
{\rm (M. G. Tkachenko)} {\it Is there a reflexive precompact group topology on a {\rm  countably} infinite Abelian group (for example, on $\mathbb{Z}$)}?
\end{problem}

Taking into consideration of Corollary \ref{c04} and Proposition \ref{p01}, one can ask:
\begin{problem}
{\it Which $MAP$ Abelian groups are $\mathfrak{g}$-closed in its Bohr compactification}?
\end{problem}

\end{document}